\documentclass[11pt]{amsart}

\usepackage{amscd,amssymb,amsopn,amsmath,amsthm,mathrsfs,graphics,amsfonts,enumerate,verbatim,calc
}

\usepackage{url}

\usepackage[OT2,OT1]{fontenc}
\newcommand\cyr{%
\renewcommand\rmdefault{wncyr}%
\renewcommand\sfdefault{wncyss}%
\renewcommand\encodingdefault{OT2}%
\normalfont
\selectfont}
\DeclareTextFontCommand{\textcyr}{\cyr}

\usepackage{amssymb,amsmath}

\DeclareFontFamily{OT1}{rsfs}{}
\DeclareFontShape{OT1}{rsfs}{n}{it}{<-> rsfs10}{}
\DeclareMathAlphabet{\mathscr}{OT1}{rsfs}{n}{it}

\numberwithin{equation}{section}
\newtheorem{theorem}{Theorem}[section]
\newtheorem{lemma}[theorem]{Lemma}
\newtheorem{proposition}[theorem]{Proposition}
\newtheorem{corollary}[theorem]{Corollary}

\newtheorem{question}{Question}

\newtheorem{claim}{Claim}
\newtheorem{notation}[theorem]{Notation}

\theoremstyle{definition}
\newtheorem{definition}[theorem]{Definition}
\newtheorem{remark}[theorem]{Remark}
\theoremstyle{remark}

\newtheorem{example}[theorem]{Example}
\newtheorem{acknowledgement}{Acknowledgement}

\newcommand{\Ass}{\operatorname{Ass}}
\newcommand{\im}{\operatorname{Im}}
\renewcommand{\ker}{\operatorname{Ker}}
\newcommand{\grade}{\operatorname{grade}}

\newcommand{\Spec}{\operatorname{Spec}}

\newcommand{\cd}{\operatorname{cd}}

\newcommand{\Ht}{\operatorname{ht}}
\newcommand{\pd}{\operatorname{pd}}

\newcommand{\V}{\operatorname{V}}

\newcommand{\Supp}{\operatorname{Supp}}

\newcommand{\Att}{\operatorname{Att}}
\newcommand{\Ann}{\operatorname{Ann}}

\newcommand{\Char}{\operatorname{char}}

\newcommand{\depth}{\operatorname{depth}}

\newcommand{\Max}{\operatorname{Max}}
\newcommand{\Min}{\operatorname{Min}}

\newcommand{\lo}{\longrightarrow}
\newcommand{\fm}{\frak{m}}
\newcommand{\fp}{\frak{p}}
\newcommand{\fq}{\frak{q}}
\newcommand{\fa}{\frak{a}}
\newcommand{\fb}{\frak{b}}

\newcommand{\fn}{\frak{n}}

\begin{document}
\title[Surjectivity of some local cohomology map and   (SVT)]{Surjectivity of some local cohomology map and the second vanishing theorem}

\author[Asgharzadeh]{Mohsen Asgharzadeh}
\address{Hakimieh, 16599-19556, Tehran, Iran}
\email{mohsenasgharzadeh@gmail.com}

\author[Ishiro]{Shinnosuke Ishiro}
\address{Department of Mathematics, College of Humanities and Sciences, Nihon University, Setagaya-ku,
Tokyo 156-8550, Japan}
\email{shinnosukeishiro@gmail.com}

\author[Shimomoto]{Kazuma Shimomoto}
\address{Department of Mathematics, College of Humanities and Sciences, Nihon University, Setagaya-ku,
Tokyo 156-8550, Japan}
\email{shimomotokazuma@gmail.com}




\begin{abstract}
The second vanishing theorem has a long history in the theory of local cohomology modules, which connects the vanishing of a complete regular local ring with a topological property of the punctured spectrum of the ring under some conditions. However, the case of complete ramified regular local rings is unresolved. In this paper, we give a partial answer to the second vanishing theorem in the ramified case. Our proof is inspired by the theory of surjective elements in the theory of local cohomology.
\end{abstract}

\maketitle


\section{Introduction}

The second vanishing theorem (abbreviated as ``(SVT)") is an assertion of the vanishing of certain local cohomology modules on a regular local ring.
Hartshorne applied the cohomology of formal schemes to present an interesting proof of (SVT) in the case of polynomial rings, and he
asked   (SVT) for general regular local rings; see \cite[Page 445]{HV}.
The equal characteristic case is well-known. Ogus \cite{Ogu73} proved (SVT) in characteristic $0$, Hartshorne and Speiser \cite{HS77} proved it in characteristic $p>0$, and Huneke and Lyubeznik \cite{HL90} gave a uniform treatment in the equal characteristic case. Recently, Zhang \cite{Zha21} proved (SVT) for unramified regular local rings. Bhattacharyyat \cite{Bha20} partially proved it for ramified regular local rings. However, it is unresolved in general. The aim of this paper is to prove the following result, which is regarded as an analogue of (SVT).

\begin{theorem}
\label{svt}
Let $(R,\fm,k)$ be a $d$-dimensional complete regular local ring of mixed characteristic with separably closed residue field $k$. Assume that $I \subset R$ is a proper ideal with $\dim(R/\fp) \ge 3$ and $\ell_R\bigl(H^{2}_\fm(R/pR+\fp)\bigr)<\infty$ for any $\fp \in \Min(R/I)$. Then the following statements are equivalent.
\begin{enumerate}
\item
$H^{d-1}_I(R)=0$.

\item
The punctured spectrum $\Spec^\circ(R/I)$ is connected in the Zariski topology.
\end{enumerate}
\end{theorem}

Here, $\ell_R(-)$ stands for the length function.
In \cite{HBPW18}, the above theorem is established in the unramified case under the assumption that $R/I$ is equidimensional instead of the finiteness condition on local cohomology modules.
In our setting, we need the condition on the finiteness of local cohomology modules and deal with unramified and ramified cases simultaneously. In order to prove Theorem 1.1, we show that a certain local cohomology map, induced by a multiplication with a suitable regular element in $R$, is surjective. For the precise statement, see Theorem \ref{surjective2}. This motivates us to discuss more on the surjectivity property of the multiplicative map $H^{i}_I(R) \xrightarrow{\times x} H^{i}_I(R)$, and open some related questions. We present some criteria for it (see Propositions \ref{regularseq}, \ref{surjective} and \ref{cr}). Some illustrative examples are also given.

Let us remark that the hypothesis in Theorem \ref{svt} applies to both unramified and ramified regular local rings.
We emphasize that this paper is guided by the following general question and we will see that the main theorem gives one such realization:

\begin{enumerate}
\item[$\bullet$]
Let $R$ be a ring of mixed characteristic $p>0$. Then $R/pR$ is a ring of prime characteristic. Can one find a ring $S$ of prime characteristic with a regular element $t \in S$ such that $S/tS \cong R/pR$ and $S$ inherits certain good properties from $R$?
\end{enumerate}

We construct some examples fitting into the setting of Theorem $\ref{svt}$. For more details, see Examples \ref{connectedexam1}, \ref{nonconnectedgraph} and \ref{32}.

In the last part of this paper, we give the mixed characteristic analogues of some results of Varbaro (see e.g. Proposition \ref{cohd}).
One may regard this as an application of 
(SVT). This enables us to calculate  the numerical invariant $q_{I}(R)$, see Proposition \ref{artininvariant}.

\subsection{Notations and conventions}

\begin{enumerate}
\item[$\bullet$]
All rings in this paper are assumed to be commutative noetherian with unity.

\item[$\bullet$]
For an ideal $I \subset R$, we denote by $\Min(R/I)$ the set of prime ideals of $R$ that are minimal over $I$.

\item[$\bullet$]
For a local ring $(R,\fm,k)$, the punctured spectrum is the set $\Spec(R) \setminus \{\fm\}$ denoted by $\Spec^{\circ}(R)$.

\item[$\bullet$]
For a ring $R$ and an element $f \in R$, $R_f = S^{-1}R$ denotes the ring of fractions of $R$ with respect to $S:=\{f^n\}_{n \geq 0}$.
\end{enumerate}

\section{Preliminaries; local cohomology modules}

We refer the reader to \cite{24hours} for local cohomology modules.

\subsection{Definition of local cohomology modules and long exact sequences}
Let $R$ be a  ring and let $I$ be an ideal generated by $f_{1},\ldots,f_{r}$ of $R$. For an $R$-module $M$, we have the \v{C}ech complex
\begin{equation}
C^{\bullet}(I;M) : 0 \lo M \lo \prod_{1 \leq j \leq r}M_{f_{j}} \lo \prod_{1 \leq j_{1} < j_{2} \leq r} M_{f_{j_{1}}f_{j_{2}}} \lo \cdots \lo M_{f_{1}\cdots f_{r}} \lo 0.
\end{equation}
Then its $i$-th cohomology is called the \textit{$i$-th local cohomology module} and denoted by $H^{i}_{I}(M)$. We often deal with the long exact sequences of three different types. First for a short exact sequence $0 \to N \to M \to L \to 0$, we obtain the long exact sequence
\begin{equation}
\label{longcoh1}
\cdots \lo H^{i}_{I}(N) \lo H^{i}_{I}(M) \lo H^{i}_{I}(L) \lo H^{i+1}_{I}(N) \lo \cdots.
\end{equation}

Take an element $f \in R$, an ideal $I \subset R$ and an $R$-module $M$. Then we obtain the long exact sequence
\begin{equation}
\label{longcoh2}
\cdots \lo H^{i}_{I+f}(M) \lo H^{i}_{I}(M) \lo H^{i}_{I}(M_{f}) \lo H^{i+1}_{I+f}(M) \lo \cdots
\end{equation}
which is induced from the short exact sequence of complexes $0 \to C^{\bullet}(I+f;M) \to C^{\bullet}(I;M) \to C^{\bullet}(I;M_{f}) \to 0$. 

Finally, take ideals $I,J \subset R$ and an $R$-module $M$. Then we obtain the long exact sequence, called the \textit{Mayer-Vietoris long exact sequence}
\begin{equation}
\label{longcoh3}
\cdots \lo H^{i}_{I+J}(M) \lo H^{i}_{I}(M) \oplus H^{i}_{J}(M) \lo H^{i}_{I \cap J}(M) \lo H^{i+1}_{I+J}(M) \lo \cdots.
\end{equation}

\begin{notation}
We set $\cd(I,M):=\sup\{i:H^{i}_{I}(M)\neq 0\}$.
\end{notation}

\subsection{Connectedness of punctured spectra}
The punctured spectrum $\Spec^{\circ}(R)$ of a local ring $(R,\fm,k)$ is the set of all primes $\fp \ne \fm$ with the topology induced by the Zariski topology on $\Spec R$. Let $I$ be an ideal of $R$. The punctured spectrum $\Spec^{\circ}(R/I)$ is \textit{connected} if the following property holds; For any ideals $\fa$ and $\fb$ of $R$ such that $\sqrt{\fa \cap \fb}=\sqrt{I}$ and $\sqrt{\fa + \fb} = \fm$, we have $\sqrt{\fa}$ or $\sqrt{\fb}$ equals $\fm$. Or equivalently, $\sqrt{\fa}$ or $\sqrt{\fb}$ equals $\sqrt{I}$.

\begin{remark}
If $R$ is a local domain, then it is easy to see that the punctured spectrum $\Spec^{\circ}(R)$ is connected.
\end{remark}

\section{Some surjective maps of local cohomology modules}

In this section, we will investigate the surjectivity of certain local cohomology maps. This is motivated by the notion of \textit{surjective elements}. While the papers \cite{HMS14} and \cite{MQ18} study the surjectivity of local cohomology maps in terms of $F$-singularities, we take another route, which is explained below.








\begin{definition}
Let $R$ be a ring and let $M$ be an $R$-module. Then $M$ is \textit{divisible} if for every regular element $r \in R$, and every element $m \in M$, there exists an element $m' \in M$ such that $m=rm'$. 
\end{definition}

Note that every injective module is divisible. Huneke and Sharp investigated the condition under which local cohomology modules of regular local rings of positive characteristic are injective (see \cite{HS93}). We recall some results which we apply for Theorem \ref{surjective2}.

\begin{lemma}{$($\cite[1.8 Lemma]{HS93}$)$}
\label{HS1.8}
Let $(R,\fm,k)$ be a regular local ring of characteristic $p>0$, and let $I$ be an ideal of $R$. Then  ${\mathscr{F}^{1}_R}(H^i_I(R)) := R^{(1)} \otimes_R H^i_I(R) \cong H^i_I(R)$ for all $i \ge 0$.
\end{lemma}

\begin{lemma}{$($\cite[3.6 Corollary]{HS93}$)$}
\label{HS3.6}
Let $(R,\fm,k)$ be a regular local ring of characteristic $p>0$, and let $M$ be an artinian $R$-module such that $\mathscr{F}^{1}_R(M) \cong M$. Then $M$ is an injective $R$-module.
\end{lemma}

The following theorem is a key in this article. We emphasize the importance of surjectivity of a map of local cohomology modules.

\begin{theorem}
\label{surjective2}
Let $(R,\fm,k)$ be a $d$-dimensional regular local ring of characteristic $p>0$. Fix an ideal $I \subset \fm$. Suppose  $\ell_R(H^{d-i}_\fm(R/I))<\infty$ for some fixed $i \ge 0$. Then $H^{i}_I(R)$ is a divisible module. In particular, for any nonzero element $x \in R$, the multiplication map $H^{i}_I(R) \xrightarrow{\times x} H^{i}_I(R)$ is surjective.
\end{theorem}

\begin{proof}
Since the $R$-module $H^{d-i}_\fm(R/I)$ is of finite length, it follows from \cite[Corollary 3.3]{Lyu06} and \cite[Corollary 3.4]{Lyu06} that $H^i_I(R)$ is an artinian $R$-module. Combining Lemma \ref{HS1.8} with this fact and Lemma \ref{HS3.6}, $H^i_I(R)$ is an injective $R$-module. Thus, it is divisible. Noting that $R$ is a domain, $H^{i}_I(R) \xrightarrow{\times x} H^{i}_I(R)$ is surjective for any nonzero element $x \in R$, as desired.
\end{proof}

A natural question arises:

\begin{question}
\label{Ques1}
Let $(R,\fm)$ be a  regular local ring of any characteristic. Is the surjectivity $H^{i}_I(R) \xrightarrow{\times x} H^{i}_I(R)$ true without assuming $\ell_R(H^{d-i}_\fm(R/I))<\infty$?
\end{question}

This is not the case even if we assume $x\in I$:

\begin{example}(\cite[Corollary 6.7]{mixedlyu})
\label{37}
There is a $d$-dimensional regular local ring $R$ of mixed characteristic
$p>0$ equipped with an ideal $I$ containing $p$, so that the map
$H^{i}_I(R) \xrightarrow{\times p} H^{i}_I(R)$ is not surjective.
\end{example}

So under what conditions does the map $H^{i}_I(R) \xrightarrow{\times x} H^{i}_I(R)$ become surjective? 


\begin{proposition}\label{regularseq}
Let $R$ be a ring and let $I$ be an ideal of $R$ generated by a regular sequence.
Then $H^{i}_I(R) \xrightarrow{\times x} H^{i}_I(R)$ is surjective for some $x \in I$.
\end{proposition}

\begin{proof}
Let $\underline{x}:=x_1,\ldots,x_m$ be a regular sequence which generates $I$. Without loss of generality, we may assume that $i=m$. 
Then considering the long exact sequence induced by 
$$
0 \lo R \xrightarrow{\times{x_1}} R \lo R/x_1R \lo 0,
$$
we obtain the desired surjection.
\end{proof}

\begin{lemma}\label{local}
Let $B$ be a noetherian ring and $H$ be a module supported at a maximal ideal $\fm\in\Max(B)$. If $H_\fm$ is finite length as a  $B_\fm$-module, then $\ell_B(H)<\infty$.
\end{lemma}

\begin{proof}Since artinian modules  are finitely embedded and that $\Supp(H)\subseteq\{\fm\}$, there is an integer $t$ such that
\begin{equation}\label{20221002}
H_{\fm}\subseteq \oplus_tE_{B_\fm}(B_\fm/ \fm B_\fm)\cong\oplus_t E_{B}(B / \fm  )_\fm=\oplus_tE_{B}(B / \fm  ).
\end{equation}
Let $f:H\to H_\fm$ be the localization map, and let $K:=\ker(f)$. We look at $$0\lo K_{\fm}\lo H_{\fm}\stackrel{f_{\fm}}\lo (H_{\fm})_{\fm}.$$Since $f_{\fm}$ is an isomorphism, we deduce $K_{\fm}=0$. Since $K\subseteq H$,  we have $\Supp(K)\subseteq \Supp(H)\subseteq\{\fm\}$. Consequently, $K=0$.
From this, $$H\subseteq H_{\fm}\stackrel{(\ref{20221002})}\subseteq \oplus_tE_{B}(B / \fm  ).$$
This shows that the $B$-module $H$ is artinian. Recall that there is an integer  $s$ such that $({\fm}B_\fm)^s H_{\fm}=0$. So, $\fm^sH\subset ({\fm}B_\fm)^s H_{\fm}=0$. As $H$ is artinian, this shows that $\ell_B(H)$ is  finite.
\end{proof}
Here is a characteristic zero version of Theorem \ref{surjective2}:

\begin{proposition}
\label{surjective}
Let $(R,\fm,k)$ be a $d$-dimensional Cohen-Macaulay  local integral domain containing the field $k$. Fix a square-free monomial ideal $I \subset \fm$ with respect to a full system of parameters (that is, $I$  is generated by
square-free monomials in a system of parameters $x_1,\ldots,x_d$). Suppose  $\ell_R(H^{d-i}_\fm(R/I))<\infty$  for some fixed $i \ge 0$. Then    $H^{i}_I(R) \xrightarrow{\times x} H^{i}_I(R)$ is surjective for  some   $x \in I$.
\end{proposition}

\begin{proof}
 By a result of Hartshorne \cite[Proposition 1]{H66}, $B:=k[\underline{x}]$ is the polynomial ring and the inclusion map  $B\hookrightarrow R$ is a flat extension.
 Let $\fm_B:=\fm\cap B$. Then $\fm_B =  (x_1,\ldots, x_d)B $ is the irrelevant maximal ideal of $B$. Let $A:=B_{\fm_B}$. According to
 \cite[Theorem 7.1]{Mat86}
 we deduce that $A\to R_{\fm}=R$ is flat. We are going to use  \cite[Theorem 7.2]{Mat86} to conclude that $A\to R$ is faithfully flat.   Let $J$ be the ideal of $B$ such that
$JR=I$. By the assumption $J$ exists, and it is a square-free monomial ideal in $B$. Let $J_0:=JA$. Then $J_0R=I$. 
We are going to use flat base change theorem along with the independence theorem for local cohomology modules to deduce that $$H^{d-i}_\fm(R/I) \cong H^{d-i}_{\underline{x}}(A/J_0)\otimes_AR\cong H^{d-i}_{\fm_A}(A/J_0)\otimes_AR,$$
where $\fm_A :=  (x_1,\ldots, x_d)A $ is the irrelevant maximal ideal of $A$.
Let $0\subsetneqq M_0\subsetneqq \ldots \subsetneqq H^{d-i}_{\fm_A}(A/J_0)$ be strict. Since
$A\hookrightarrow R$ is  faithfully flat, $$0\subsetneqq M_0\otimes_AR\subsetneqq \ldots \subsetneqq H^{d-i}_{\fm_A}(A/J_0)\otimes_AR.$$
From this, $\ell_A(H^{d-i}_{\fm_A}(A/J_0))<\infty$.
Recall that local cohomology modules behave  well with respect to localization, and that $\Supp_B\left(H^{d-i}_{\fm_B}(B/J)\right)\subseteq \{\fm_B\}$.
Combine these along with  Lemma \ref{local},
we know that
$\ell_B(H^{d-i}_{\fm_B}(B/J))<\infty$.
This allows us to apply a result of Richardson \cite[Theorem 2.4]{rich}, and observe that $$\Att_B(H^{d-i}_{\fm_B}(B/J))=
\Ass_B(H^{i}_J(B)).$$ Since $H^{d-i}_{\fm_B}(B/J)$ is of finite length, we have $\Att_B(H^{d-i}_{\fm_B}(B/J)) \subset \{\fm_B\}$ and thus, $\Ass_B(H^{i}_J(B))\subset\{\fm_B\}$. It follows that the dimension of $\Supp_B(H^{i}_J(B))$ is zero, provided $H^{i}_J(B)\neq 0$. Without loss of generality, we may assume $H^{i}_J(B)\neq 0$, and that $\Char(k)=0$. By a result of Lyubeznik \cite[Corollary 3.6(b)]{Lyu93}, we know $H^i_J(B)$ is an injective $B$-module. Noting that $B$ is a domain, $H^{i}_J(B) \xrightarrow{\times x} H^{i}_J(B)$ is surjective for any nonzero element $x \in B$. Since $H^{i}_I(B) \cong  H^{i}_J(B)\otimes_BR$,
 we get the desired claim.
\end{proof}

\begin{example}
Let $I$ be an ideal of a noetherian integral domain $R$ and $M$ be any $R$-module. Suppose $I=(x_1,\ldots,x_n)$. Then $H^{n}_I(M) \xrightarrow{\times x_i} H^{n}_I(M)$ is surjective for all $1 \leq i \leq n$.
\end{example}

\begin{proof}
First, we deal with the case $M:=R$. We look at the exact sequence $0\to R \xrightarrow{\times x_i} R\to R/x_i R\to 0$, and deduce the following
exact sequence:
$$
\zeta:=\big\{H^{n}_I(R) \xrightarrow{\times x_i} H^{n}_I(R)\lo H^{n}_I(R/x_{i}R)\big\}.
$$
Let $J:=(x_1,\ldots,\widehat{x_i},\ldots, x_{n})$, and recall that its cohomological dimension is at most $n-1$. Since
$x_i$ annihilates $H^{n}_J(R/x_iR)$, it follows from $(\ref{longcoh2})$ that
$$
0=H^{n}_J(R/x_iR)=H^{n}_{J+x_iR}(R/x_iR)=H^{n}_I(R/x_iR).
$$
We use this along with the exact sequence $\zeta$, to conclude that $H^{n}_I(R) \xrightarrow{\times x_i}   H^{n}_I(R)$ is surjective.

Recall that the functor $H^{n}_I(-)$ is right-exact. Now, let $M$ be an $R$-module. Since $H^{n}_I(M) \cong H^{n}_I(R)\otimes _RM$, we deduce that $H^{n}_I(M) \xrightarrow{\times x_i} H^{n}_I(M)$ is surjective.
\end{proof}

\begin{corollary}
\label{cn}
Let $\fp$ be a prime ideal of the ring of formal power series over a perfect field $R := k[|X_1,\ldots,X_n|]$ of height $n-1$. Then $H^{i}_\fp(R) \xrightarrow{\times x} H^{i}_\fp(R)$ is surjective for some $x \in \fp$ and all $i$.
\end{corollary}

\begin{proof}
The only nontrivial $i$ is $n-1$ by the Hartshorne-Lichtenbaum's vanishing theorem and the cohomological characterization of grade. By a celebrated result of Cowsik-Nori \cite{cn}, there are $f_1,\ldots,f_{n-1}$ such that $\sqrt{(f_1,\ldots,f_{n-1})}=\fp$. Since local cohomology does not change up to the radical of ideals, it is enough to apply the previous example.
\end{proof}

\begin{corollary}\label{ccn}
Adopt the notation of Corollary \ref{cn}, and let $M$ be any $R$-module. Then the map $H^{n-1}_\fp(M) \xrightarrow{\times x} H^{n-1}_\fp(M)$ is surjective for some $x \in \fp$.
\end{corollary}

\begin{proof}
Since cohomological dimension of $\fp$ is $n-1$, we know that $H^{i}_\fp(M)=0$ for all $i>n-1$. From this, we observe that  $H^{n-1}_\fp(M)\cong H^{n-1}_\fp(R)\otimes_RM$. It remains to apply Corollary \ref{cn}.	
\end{proof}

\begin{remark}
\label{attachedprime}
Let $(R,\fm)$ be a local ring and let $0 \ne M$ be a finitely generated module of dimension $n$. Then there is an element $r\in\fm$ such that $H^n_{\fm}(M) \xrightarrow{\times r} H^n_{\fm}(M)$ is surjective if and only if $n>0$.

Indeed, if $n=0$, then the claim is clear by Nakayama's lemma. So we may assume that $n>0$. By a computation of Macdonald and Sharp \cite{Msh7}, we know that
$$
\Att(H^n_{\fm}(M))=\{\fp\in\Ass(M)~|~\dim(R/ \fp)=n\}.
$$
Then by prime avoidance, we get $\fm\nsubseteq \bigcup _{\fq\in\Att H^n_{\fm}(M)}\fq$. This enables us to find $r\in\fm\setminus\bigcup_{\fq\in\Att H^n_{\fm}(M)}\fq$. Now, the desired claim is clear by the definition of attached prime ideals.
\end{remark}

Adopt the notation of Example \ref{37}. It is easy to see that 
 $H^3_{I}(R/(p^n))\to H^3_{I}(R/(p)) $ is not surjective (see the proof of \cite[Proposition 6.3]{mixedlyu}). This suggests the following question:

\begin{question}
\label{Ques2}
Do the analogues of \cite[Proposition 3.2]{MQ18} and \cite[Proposition 3.3]{MQ18} hold for local cohomology modules supported by an arbitrary ideal?
\end{question}

In the rest of this section, let us consider the following question:

\begin{question}
\label{Ques3}
Let $x \in \fm$ be a regular element in a local ring $(R,\fm)$. Consider the following two properties:
\begin{enumerate}
\item[i)]
The natural maps $H^i_{I}(R/x^nR)\to H^i_{I}(R/xR) $, induced by $R/x^nR \to R/xR$, is surjective for all $i\geq0$ and $n>0$.

\item[ii)]
The multiplication map $H^i_{I}(R)\xrightarrow{\times x} H^i_{I}(R) $ is surjective for all $i\geq0$.
\end{enumerate}
Does $i)$ imply $ii)$?
\end{question}

Let us show that the regularity assumption on $x$ is important:

\begin{example}
Let $(A,\fn)$ be a regular local ring with two nonzero prime ideals $\fp:=xA$ and $\fq$ such that $\fp\neq \fq$. Denote the ring $A/\fp\fq$ by $R$. Note that $x$ is not regular in $R$. Then we claim that $H^i_{\fp}(R/x^nR) \longrightarrow H^i_{\fp}(R/xR)$ is surjective for all $i\geq 0$ and $n>0$, while $H^i_{\fp}(R) \xrightarrow{\times x} H^i_{\fp}(R) $ is not surjective. Indeed the map
$$
\begin{CD}
H^0_{\fp}(R) @>\times x>> H^0_{\fp}(R) \\
=@AAA= @AAA \\
\fq/\fp\fq @>\times x>> \fq/\fp\fq \\
\end{CD}
$$
is not surjective by Nakayama's lemma. But, we have the following diagram with exact rows:
$$
\begin{CD}
H^0_{\fp}(R/x^nR)	@>>> H^0_{\fp}(R/xR) \\
=@AAA= @AAA \\
R/x^nR @>>> R/xR @>>>0\\
\end{CD}
$$
Also, $H^i_{\fp}(R/x^nR)=H^i_{\fp}(R/xR)=0$ for all $i>0$ and $n>0$, because $R/x^nR$ is $\fp$-torsion. In sum, $H^i_{\fp}(R/x^nR) \longrightarrow H^i_{\fp}(R/xR)$ is surjective for all $i\geq 0$ and $n>0$.
\end{example}

\begin{example}
Set $R:=k[| x,y |]$ and $\fp:=(y)$. Then the following assertions hold:
\begin{enumerate}
\item[i)]
The natural map $H^i_{\fp}(R/x^nR)\to H^i_{\fp}(R/xR) $, which is induced by $R/x^nR \to R/xR$, is surjective for all $i\geq0$ and $n>0$.

\item[ii)]
The multiplication map $H^1_{\fp}(R)\xrightarrow{\times x} H^1_{\fp}(R) $ is not surjective.
\end{enumerate}
\end{example}

\begin{proof}i)
Recall that $H^2_{\fp}(R/x^nR)=H^2_{\fp}(R/xR)=0 $. Also,  $H^0_{\fp}(R/x^nR)=H^0_{\fp+(x^n)}(R/x^nR)=H^0_{\fm}(R/x^nR)=0 $, because $x^n,y$ is a regular sequence of length two. After localizing $R/x^nR\to R/xR\to 0$ at $y$, we get the exact sequence $R_y/x^nR_y\to R_y/xR_y\to 0$. Now, we deduce from the following diagram
$$
\begin{CD}
@. 0 @.0 @. \\
@. @AAA  @AAA    \\
0@>>> R/xR @>>> R_y/xR_y @>>> H^1_{\fp}(R/xR)@>>> 0\\
@. @AAA  @AAA f  @AAA   \\
0@>>> R/x^nR @> >> R_y/x^nR_y @>>> H^1_{\fp}(R/x^nR)@>>> 0,\\
\end{CD}
$$
that $f$ is surjective.

ii) We look at
$$
\begin{CD}
H^1_{\fp}(R) @>x>> H^1_{\fp}(R)\\
=@AAA= @AAA\\
\frac{R_y}{R}	 @>\times x>>  \frac{R_y}{R},\\
\end{CD}
$$
and we claim that the bottom map is not surjective. Suppose on the way of contradiction that $1/y+R$ is in the image, i.e., there is an $r\in R$ and an $n\in\mathbb{N}_0$ such that $x(r/y^n+R)=1/y+R$. In other words, there is an $s\in R$ such that $\frac{rx-y^{n-1}}{y^n}=s$. This in turn is equivalent to $rx=y^{n-1}(1-sy)$. By the UFD property of $R$, we get $x|(1-sy)$, a contradiction.
\end{proof}

This suggests a little more modification of Question \ref{Ques1} via imposing more restrictions on $x$:

\begin{question}
Let $x$ be a regular element in $I$. Suppose the natural map $H^i_{I}(R/x^nR)\to H^i_{I}(R/xR) $, which is induced by $R/x^nR \to R/xR$, is surjective for all $i\geq0$ and $n>0$. Then is $H^i_{I}(R)\xrightarrow{\times x} H^i_{I}(R) $ surjective?
\end{question}

By using an idea taken from \cite{MQ18}, we observe:

\begin{proposition}
\label{cr}
Suppose that the local ring $R$ is of prime characteristic $p$, $x\in I$ is a regular element and $H^{i}_{I}(R/x^{p^e}R) \to H^i_I(R/xR)$ is surjective for some $i\geq0$. Then
$H^{i+1}_{I}(R) \xrightarrow{\times x} H^{i+1}_{I}(R)$ is surjective.
\end{proposition}

\begin{proof} 
Note that $R/x^{p^{e}}R \twoheadrightarrow R/xR$ factors through the surjective map $R/x^{h}R \twoheadrightarrow R/xR$ for any $1 \leq h \leq p^{e}$. Then the induced map $H^{i}_{I}(R/x^{p^{e}}R) \to H^{i}_{I}(R/xR)$ factors through the surjective map $H^{i}_{I}(R/x^{h}R) \to H^{i}_{I}(R/x R)$. Next we show the following claim.

\begin{claim}
\label{subclaim1}
For each $l >0$ and $j \geq l$, the multiplication map $R/x^{l}R \xrightarrow{\times x^{j-l}} R/x^{j}$ induces an injective map $H^{i+1}_{I}(R/x^{l}R) \xrightarrow{\times x^{j-l}} H^{i+1}_{I}(R/x^{j}R)$.
\end{claim}

\begin{proof}[Proof of Claim 1]
Indeed, it suffice to show by induction the case $j=l+1$. The short exact sequence
$$
0 \lo R/x^{l}R  \xrightarrow{\times x} R/x^{l+1}R \lo R/xR \lo 0
$$
induces the long exact sequence of local cohomology modules
$$
\cdots \lo H^{i}_{I}(R/x^{l+1}R) \stackrel{\phi_{i}}\lo H^{i}_{I}(R/xR) \stackrel{\delta_{i}}\lo H^{i+1}_{I}(R/x^{l}R) \lo H^{i+1}_{I}(R/x^{l+1}R) \lo \cdots.
$$
Then since $\phi_{i}$ is surjective, $\delta_{i}$ is a zero map. So $H^{i+1}_{I}(R/x^{l}R) \to H^{i+1}_{I}(R/x^{l+1}R)$ is injective.
\end{proof}

Since $x$ is regular, the short exact sequence $0 \to R \xrightarrow{\times x} R \to R/xR \to 0$ induces the long exact sequence of local cohomology modules
$$
\cdots \lo H^{i+1}_{I}(R) \xrightarrow{\times x} H^{i+1}_{I}(R) \xrightarrow{\pi_{i+1}} H^{i+1}_{I}(R/xR) \stackrel{\delta} \lo {} H^{i+2}_{I}(R) \lo \cdots.
$$
By Claim \ref{subclaim1}, we obtain an injection $H^{i+1}_{I}(R/x^{l}R) \xrightarrow{\times x} H^{i+1}_{I}(R/x^{l+1}R)$ for any $l>0$. This implies that $\{H^{i+1}_{I}(R/x^k R)\}_{k \geq 0}$ forms an injective direct system.  Then 	
\begin{equation}
\label{localExactSeq1}
H^{i+1}_I(R/xR) \hookrightarrow
\varinjlim_n H^{i+1}_I(R/x^nR)  \stackrel{\cong}\lo  H^{i+1}_I(H^{1}_x(R)).
\end{equation}
Recall that $0\to R\to R_x\to H^{1}_x(R)\to 0$ induces
$$
H^{i}_I(R_x)\lo H^{i}_I(H^{1}_x(R)) \lo H^{i+1}_I(R)\lo H^{i+1}_I(R_x).
$$
Since $x\in I$, it follows that $H^{i+1}_I(R_x)=H^{i}_I(R_x)=H^{i}_{IR_x}(R_x)=H^{i}_{R_x}(R_x)=0$, and so $H^{i}_I(H^{1}_x(R)) \cong H^{i+1}_I(R)$. Combining this with $(\ref{localExactSeq1})$, we observe that $H^{i+1}_I(R/xR) \xrightarrow{\delta} H^{i+2}_I(R)$ is injective. Therefore, we obtain $\im\pi_{i+1}=\ker\delta=0$. This implies that $H^{i+1}_{I}(R) \xrightarrow{\times x} H^{i+1}_{I}(R)$ is surjective, as claimed.
\end{proof}

\section{Application to vanishing of local cohomology modules}

\subsection{Proof of Theorem \ref{svt}}

First we define two graphs. These are important to prove the main theorem.

\begin{definition}\label{graph}
Let $(R,\fm ,k)$ be a local ring.
\begin{enumerate}
\item
Let $\fp_1,\ldots,\fp_t$ be the set of minimal primes of $R$. Then the graph $\Theta_R$ has vertices labeled $1,\ldots,t$ and there is an edge between two (distinct) vertices $i$ and $j$, precisely when $\fp_i+\fp_j$ is not $\fm$-primary.

\item(\cite[Definition 3.4]{HH94})
Let $\fp_{1}, \ldots, \fp_{r}$ be the set of minimal primes of $R$ such that $\dim(R/\fp_{i})=\dim(R)$. Then the Hochster-Huneke graph of $R$, which is denoted by $\Gamma_{R}$, has vertices $1,\ldots,r$ and there is an edge between two (distinct) verticies $i$ and $j$, precisely if $\fp_{i}+\fp_{j}$ has height one.
\end{enumerate}
\end{definition}

Huneke and Lyubeznik pointed out the importance of the graph $\Theta_{R}$ when we investigate the connectedness of punctured spectra.

\begin{lemma}{$($\textrm{Huneke-Lyubeznik}$)$}
\label{graph1}
Let $(R,\fm,k)$ be a local ring. Then $\Spec^{\circ}(R)$ is connected if and only if  $\Theta_R$ is connected.
\end{lemma}

The following lemma is important to prove Lemma \ref{vanishingmixed}.
\begin{lemma}
\label{graph2}
Let $(R,\fm,k)$ be a complete local domain and let $x \in \fm$ be a nonzero element. Then  $\Gamma_{R/xR}$  is connected.
\end{lemma}
\begin{proof}
This is in \cite[Proposition 3.5]{HBPW18}  with some contributions to \cite[Proposition 2.2]{Zha} and \cite[Expos\'{e} XIII, Th\'{e}or\`{e}me 2.1]{sga2} by Grothendieck.
\end{proof}

We prove the following lemma which is similar to \cite[Lemma 3.7]{HBPW18}, but the proof is different in nature.

\begin{lemma}
\label{vanishingmixed}
Let $(R,\fm,k)$ be a d-dimensional complete regular local ring of mixed characteristic with separably closed residue field $k$. Let $\fq$ be a prime ideal of $R$ such that $\dim(R/\fq) \ge 3$. Put $J:=pR+\fq$. Assume $\ell_R(H^{2}_\fm(R/J))<\infty$. Then $H_{\fq}^{d-1}(R)=0$.
\end{lemma}

\begin{proof}
Since $R$ is a complete regular local ring of mixed characteristic, it is isomorphic to 
$$
C(k)[|x_{1}, \ldots ,x_{d}|]/(p-f)
$$
where $f$ is an element in $(p, x_1,\cdots,x_d)^2 \backslash pR$ or $f=x_{1}$, and $C(k)$ is a complete discrete valuation ring such that $C(k)/pC(k) \cong k$.

First consider the case $p \in \fq$. Put $S := k[|x_{1}, \ldots , x_{d}|]$. Then we have the isomorphism $R/pR \cong S/\overline{f}S$ where $\overline{f}$ is the image of $C(k)[|x_{1},\ldots,x_{d}|] \twoheadrightarrow k[|x_{1},\ldots,x_{d}|]$. Thus we obtain the short sequence
\begin{equation}
\label{shortexactsequence1}
0 \lo S \xrightarrow{\times \overline{f}} S \lo R/pR \lo 0.
\end{equation}
Applying the long exact sequence for (\ref{shortexactsequence1}), we get
\begin{equation}
\label{longexactsequence1}
\cdots \lo H^{i}_{\fp}(S) \xrightarrow{\times \overline{f}} H^{i}_{\fp}(S) \lo H^{i}_{\fq}(R/pR) \lo H^{i+1}_{\fp}(S) \lo \cdots,
\end{equation}
where $\fp \subset S$ is the inverse image of the ideal $\fq(R/pR)$ under $S \twoheadrightarrow R/pR$. Let us show the vanishing $H^{d-2}_{\fq}(R/pR) = 0$. Since $p$ is contained in $\fq$ and $\fq$ is a prime ideal, the punctured spectrum of $(R/pR)/(\fq/pR) \cong R/\fq$ is connected. In addition, we have the inequality $\dim (R/\fq) \geq 2$. So we obtain $H^{d-1}_{\fp}(S)=0$ by (SVT) (see  \cite[Theorem 2.9]{HL90}). Moreover, since $S/\overline{f}S \cong R/pR$ and $\overline{f} \in S$ is a nonzero element by assumption, it follows from Theorem \ref{surjective2} that the multiplication map $H^{d-2}_{\fp}(S) \xrightarrow{\times \overline{f}} H^{d-2}_{\fp}(S)$ in $(\ref{longexactsequence1})$ is surjective. To summarize the above, we obtain
\begin{equation}
\label{longexactsequence2}
\cdots \lo H^{d-2}_{\fp}(S) \xrightarrow{\times \overline{f}} H^{d-2}_{\fp}(S) \lo H^{d-2}_{\fq}(R/pR) \lo H^{d-1}_{\fp}(S)=0.
\end{equation}
Remarking that $H^{d-2}_{\fp}(S) \to H^{d-2}_{\fq}(R/pR)$ is trivial by the surjectivity of $H^{d-2}_{\fp}(S) \xrightarrow{\times \overline{f}} H^{d-2}_{\fp}(S)$, we see that $H^{d-2}_{\fq}(R/pR) \to H^{d-1}_{\fp}(S)=0$ is injective. This implies that $H^{d-2}_{\fq}(R/pR) = 0$. Finally, attached to the short exact sequence: $0 \to R \xrightarrow{\times p} R \to R/pR \to 0$, we get the exact sequence
\begin{equation}
\label{shortexactsequence2}
0 \lo H^{d-1}_\fq(R) \xrightarrow{\times p} H^{d-1}_\fq(R).
\end{equation}
Assume that $H^{d-1}_\fq(R) \neq 0$. Since $H^{d-1}_\fq(R)$ is $\fq$-torsion and $p \in \fq$, every element of $H^{d-1}_\fq(R)$ is annihilated by $p^n$ for some $n>0$. This gives a contradiction to the injectivity of $H^{d-1}_\fq(R) \xrightarrow{\times p} H^{d-1}_\fq(R)$. Thus we obtain the vanishing $H^{d-1}_\fq(R) = 0$.

Next we consider the case $p \notin \fq$. For distinct minimal primes $\fp_{1}$ and $\fp_{2}$ of $\Spec(R/J)$, if $\Ht_{R/J}(\fp_{1} + \fp_{2})=1$, then $\fp_{1} + \fp_{2}$ can not be primary to the maximal ideal of $R/J$, since we have the inequality $\dim(R/J) \geq 2$. This implies that  $\Gamma_{R/J}$ is a subgraph of $\Theta_{R/J}$. Moreover, the two graphs $\Gamma_{R/J}$ and $\Theta_{R/J}$ have the same vertices. Indeed, since $R/\fq$ is a complete local domain which is catenary, $R/J$ is equidimensional. 

Since $R/\fq$ is a complete local domain,   $\Gamma_{R/J}$ is connected by Lemma $\ref{graph2}$ by taking $x=p$. Since $\Gamma_{R/J}$ and $\Theta_{R/J}$ have the same vertices, $\Theta_{R/J}$ is also connected. This implies that $\Spec^{\circ}(R/J)$ is connected by Lemma $\ref{graph1}$. Remark that $R/J \cong S/\overline{J}$, where $\overline{J} \subset S$ is the inverse image of the ideal $J(R/pR)$ under $S \twoheadrightarrow R/pR$. and $\dim(R/J) \geq 2$. In view of (SVT) \cite[Theorem 2.9]{HL90},   we observe that $H^{d-1}_{\overline{J}}(S)=0$. Applying the above discussion replacing $\fp$ for $\overline{J}$, we obtain $H^{d-1}_{J}(R)=0$.

Finally, recall from (4.2) the following long exact sequence
$$
\cdots \lo H^{d-1}_{J}(R) \lo H^{d-1}_{\fq}(R) \lo H^{d-1}_{\fq}(R_{p}) \lo \cdots.
$$
We already proved $H^{d-1}_{J}(R)=0$. So it suffices to show that $H^{d-1}_{\fq} (R_{p})=0$. Since $p$ is in the maximal ideal $\fm$, we obtain $\dim (R_{p})=d-1$ and $\dim(R_{p}/\fq R_{p}) \geq 2$. It then suffices to show that the localization of  $H^{d-1}_{\fq}(R_{p})$ at every prime $\fp \in \Spec R$ vanishes, where $p \notin \fp$, $\fq \subset \fp$ and $\dim (R_\fp) = d-1$. Notice that $(R_{p})_{\fp} \cong R_{\fp}$ is a regular local ring of dimension $d-1$ and $\dim(R_\fp/\fq R_\fp) \ge 2$. We obtain $(H^{d-1}_{\fq}(R_{p}))_{\fp} \cong H^{d-1}_{\fq R_{\fp}}(R_{\fp})=0$ by the Hartshorne-Lichtenbaum vanishing theorem. Thus, $H^{d-1}_{\fq}(R_{p})=0$.
\end{proof}

Now we prove Theorem \ref{svt}. The proof is based on \cite[Theorem 3.8]{HBPW18}.

\begin{proof}[Proof of Theorem 1.1]
Suppose that $H^{d-1}_{I}(R)=0$. If $\Spec^{\circ}(R/I)$ is not connected, there exist ideals $J_1,J_2$ of $R$ such that $\sqrt{J_1+J_2} = \fm$ and $\sqrt{J_1 \cap J_2} = \sqrt{I}$, but $\sqrt{J_1}$ and $\sqrt{J_2}$ are not equal to both $\fm$ and $\sqrt{I}$. By $(\ref{longcoh3})$, we get 
$$
\cdots \lo H^{d-1}_{I}(R) \lo H^{d}_{\fm}(R) \lo H^{d}_{J_1}(R) \oplus H^{d}_{J_2}(R) \lo \cdots.
$$
Then we have $H^{d-1}_{I}(R)=0$ by assumption and $H^{d}_{J_1}(R) = H^{d}_{J_2}(R) = 0$ by the Hartshorne-Lichtenbaum vanishing theorem. This is a contradiction for $H^{d}_{\fm}(R) \neq 0$ by the Grothendieck's vanishing theorem.

Conversely, suppose that $\Spec^\circ(R/I)$ is connected. We proceed by induction on  $t:=|\Min(R/I)|$. The case that $t=1$ was established in Lemma \ref{vanishingmixed}. Assume that the implication holds for all ideals $\fa$ of $R$ with $t-1 \geq 0$ minimal primes, for which $\dim(R/\fp) \geq 3$ for some minimal primes $\fp$ of $\fa$.

Fix an ideal $I$ with $t$ minimal primes such that $\dim(R/\fp) \geq 3$ for any minimal prime $\fp$ of $I$, and for which $\Spec^\circ(R/I)$ is connected. Then $\Theta_{R/I}$ is also connected by Lemma \ref{graph1}. Thus, there is an ordering $\fq_1,\fq_2,\ldots,\fq_t$ of the minimal primes of $I$ such that the induced subgraph of $\Theta_{R/I}$ on $\fq_1,\fq_2,\ldots,\fq_i$ is connected for all $1 \leq i \leq t$. This means that given $1 \leq i \leq t$, if $J_i=\fq_1 \cap \cdots \cap \fq_i$, then  $\Theta_{R/J_i}$ is connected. So we deduce that $\Spec^{\circ}(R/J_i)$ is connected. Apply $(\ref{longcoh3})$ associated to $J_{t-1}$ and $\fq_t$ to get
$$
\cdots \lo H^{d-1}_{J_{t-1}}(R) \oplus H^{d-1}_{\fq_t}(R) \lo H^{d-1}_{I}(R) \lo H^{d}_{J_{t-1}+\fq_t}(R) \lo 0.
$$
By the inductive hypothesis, $H^{d-1}_{J_{t-1}}(R)=0$. In addition, $H^{d-1}_{\fq_t}(R)=0$ by Lemma \ref{vanishingmixed}. Moreover, since the punctured spectrum of $R/I$ is connected, $\sqrt{J_{t-1}+\fq_t} \subsetneqq \fm$, so that $H^{d}_{J_{t-1}+\fq_t}(R)=0$ by the Hartshorne-Lichtenbaum vanishing theorem, which proves $H^{d-1}_{I}(R)=0$.  
\end{proof}

\subsection{Examples}

Here, we construct some examples fitting into the setting of Theorem $\ref{svt}$.

\begin{example}
\label{connectedexam1}
Let $(R,\fm)$ be a local complete generalized Cohen-Macaulay integral domain in mixed characteristic with dimension at least four. 
By Cohen's structure theorem, there is a regular local ring $A$ together with an ideal $I$ such that $R=A/I$. Since $R$ is an integral domain, the ideal must be prime. Put $J:=p + I$. If $p$ is contained in $I$, then $\ell_R(H^2_\fm(A/J))$ is obviously finite (more precisely is zero) since $I=J$ and $R$ is generalized Cohen-Macaulay of dimension at least four.
If $p$ is not contained in $I$, then $p$ is a non-zero element in $A/I$.
We look at the short exact sequence $0 \to R \xrightarrow{\times p} R\to A/J \to 0$, and the induced   long exact sequence: $$\cdots \lo H^2_\fm(R) \stackrel{g}\lo H^2_\fm(A/J) \stackrel{f}\lo H^3_\fm(R) \lo \cdots.$$
We define
\begin{itemize}
\item[$\bullet$] $K:=\ker(f)$
\item[$\bullet$] $L:=\im(f)$
\end{itemize}
This fits in the following short exact sequence 
\begin{equation}\label{1003927}
0 \lo K \lo H^2_\fm(A/J) \lo L \lo 0.
\end{equation}
Note that the value of the length function of $R$ is equal to that of $S$.
Since $L \subseteq H^3_\fm(R)$, we know that $\ell_A(L)<\infty$.  Recall that $K=\ker(f)=\im(g) $ and that $H^2_\fm(A/I)\stackrel{g}\lo \im(g)\to 0$. We combine these along with $\ell_A( H^2_\fm(R))<\infty $, and deduce that $\ell_A( K)<\infty $. 
By plugging these in $(\ref{1003927})$, we observe that  $\ell_A(H^2_\fm(A/J))$ is finite. In particular, we are in the situation of Lemma \ref{vanishingmixed}. So we have $H^{d-1}_I(A) = 0$.

\end{example}

Next, we present the more explicit examples as follows:

\begin{example}
\label{connectedexam2}
Let $k$  and $W(k)$ be as before.
Let $n \geq 2$ and $d_1,\ldots,d_n \geq 3$. Define:
\begin{itemize}
\item[a)]
$A:=W(k)[|X_{i,j}~|~1 \leq i \leq n, 1 \leq j \leq d_i|]$,

\item[b)]
$I :=\displaystyle\bigcap_{1\leq i\leq n}
(X_{1,1},\ldots, X_{1,d_{1}-1},X_{2,1}, \ldots, X_{2,d_2-1}, \ldots, X_{i-1,d_{i-1}-1}, X_{i+1,1}, \ldots, X_{n,d_{n}-1} )  \subset A$

\item[c)]
$f \in I$ : a non-zero element.
\end{itemize}

Let $R := A/(p-f)$. Then the following assertions hold:
\begin{enumerate}
\item[i)]
$R$ is a ramified regular local ring of dimension $d:=\sum^n_{i=1}d_i$.

\item[ii)]
$\dim (R/\fq) \geq 3$ and $\ell_R(H^2_\fm(R/(pR+\fq)))=0<\infty$ for all $\fq \in \Min(R/IR)$.

\item[iii)]
$\Spec^{\circ}(R/IR)$ is connected.

\item[iv)]
$H^{d-1}_I(R) = 0$.
\end{enumerate}
\end{example}

\begin{proof}
i)
Due to the relation $p-f=0$, we know the maximal ideal of $R$ is generated by the set $\{X_{i,j}~|~1 \leq i \leq n, 1 \leq j \leq d_i\}$. Then as $$d\geq\mu(\fm)\geq \dim (R)=\dim(A)-1= d,$$ we have
$\mu(\fm)=\dim (R)$. In other words, 
$R$ is regular and of dimension $d$. Since  $p-f=0$ and $f\in\fm^2$, $R$ is ramified.

ii)
Since $f$ is part of monomials appearing in the generating set of $I$, we obtain an isomorphism $R/(pR+I) \cong R/IR$. Set 
$$
\fq_{i} := (X_{1,1},\ldots, X_{1,d_{1}-1},X_{2,1}, \ldots, X_{2, d_2-1}, \ldots, X_{i-1,d_{i-1}-1}, X_{i+1,1}, \ldots, X_{n, d_{n}-1} ) .
$$
Then
$$
\Min(R/IR)=\{ \fq_i  ~|~  1\leq i \leq n\}.
$$
Now, we deduce from
$$
R/pR+\fq_i  \cong R/\fq_i \cong k[| X_{1, d_{1}}, X_{2, d_2}, \ldots ,X_{i-1, d_{i-1}} X_{i, 1}, \ldots, X_{i, d_i},X_{i+1, d_{i+1}}, \ldots  ,X_{n, d_n}    |]
$$
that $H^{2}_{\fm}(R/pR+\fq_{i})=0$ for any $\fq_{i} \in \Min(R/IR)$, because $\dim (R/pR + \fq_{i}) = \dim (R/\fq_{i}) \geq 3$ and it is regular. In particular, it is of finite length.

iii)
For any distinct two primes $\fq_i, \fq_j \in \Min(R/I)$, it is obvious that $\fq_i+\fq_j \neq \fm$.
This implies that the graph $\Theta_{R/IR}$ is connected, hence $\Spec^\circ(R/IR)$ is connected.

iv)
Apply part iii) along with Theorem \ref{svt} to deduce $H^{d-1}_I(R)  = 0$.
\end{proof}


\begin{example}
\label{nonconnectedgraph} 
Let $k$  and $W(k)$ be as before.
Let $n \geq 2$ and $d_1,\ldots,d_n \geq 3$. Define:
\begin{itemize}
\item[a)]
$A:=W(k)[|X_{i,j}~|~1 \leq i \leq n, 1 \leq j \leq d_i|]$,

\item[b)]
$I:=\displaystyle\bigcap_{1\leq i \leq n}
(X_{1,1},\ldots, X_{1,d_{1}},X_{2, 1}, \ldots, X_{2, d_2}, \ldots, X_{i-1, d_{i-1}}, X_{i+1,1}, \ldots, X_{n,d_{n}} )  \subset A$


\item[c)]
$f \in I$ : a non-zero element.
\end{itemize}

Let $R := A/(p-f)$. Then the following assertions hold:
\begin{enumerate}
\item[i)]
$R$ is a ramified regular local ring and of dimension $d:=\sum^n_{i=1}d_i$.

\item[ii)]
$\dim (R/\fq) \geq 3$ and $\ell_R(H^2_\fm(R/(pR+\fq)))=0<\infty$ for all $\fq \in \Min(R/IR)$.

\item[iii)]
$\Spec^{\circ}(R/IR)$ is not connected.

\item[iv)]
$H^{d-1}_I(R) \neq 0$.
\end{enumerate}

\begin{proof}
i)
This follows from the similar reason as in the part i) of Example \ref{connectedexam2}.

ii)
Since $f$ is part of monomials appearing in the generating set of $I$, we obtain the isomorphisms $R/(pR+I) \cong R/IR$. 
Let $$\fq_{i} := (X_{1,1},\ldots, X_{1,d_{1}},X_{2, 1}, \ldots, X_{2, d_2}, \ldots, X_{i-1, d_{i-1}}, X_{i+1,1}, \ldots, X_{n,d_{n}} ),$$ and recall that
$$
\Min(R/IR)=\{ \fq_i  ~|~  1\leq i \leq n\}.
$$
Then we deduce from
$$
R/(pR+\fq_{i}) \cong R/\fq_{i} \cong k[|   X_{i ,1}, \ldots, X_{{i}, d_{i}}     |]
$$
that $H^{2}_{\fm}(R/pR+\fq_{i})=0$ for any $\fq_{i} \in \Min(R/IR)$, because $\dim( R/pR+\fq_{i}) \geq 3$ and it is regular. In particular, it is of finite length.
Moreover, we obtain $\dim (R/\fq_{i}) \geq 3$.

iii)
For any distinct two primes $\fq_i, \fq_j \in \Min(R/IR)$, it is obvious that $\fq_i+\fq_j = \fm$.
This implies that the graph $\Theta_{R/IR}$ is not connected, hence $\Spec^\circ(R/IR)$ is not connected.

iv)
Apply part iii) along with Theorem \ref{svt} to deduce $H^{d-1}_I(R)  \neq 0$.
\end{proof}
\end{example}

\begin{remark}
Adopt the notation of Example \ref{nonconnectedgraph}. 
Suppose that $n=2$, $d_1= d_2 = 4$ and $f= x_{1i}x_{2j}$ for some $1 \leq i, j \leq 4$.

\begin{itemize}
\item[i)]
We claim that $H^{d-1}_I(R)$ is the injective envelop of $k$. In particular, $H^{d-1}_I(R) \neq 0$.

\item[ii)]
Let $R$ be an analytically unramified quasi-Gorenstein local ring of dimension $d$ together with an ideal $I$ such that  $\Spec^{\circ}(R/I)$ is not connected. Then we claim that $\Ann_R(H^{d-1}_I(R))=0$. In particular, $H^{d-1}_I(R)\neq0$.\footnote{This part is valid without any use of quasi-Gorenstein assumption. Indeed, apply the Grothendieck's non-vanishing theorem along with the displayed exact sequence.}
\end{itemize}

\begin{proof}
i): Recall from the relation $p-x_iy_j=0$ that $\fq_1+\fq_2=\fm$. Also as $R$ is regular, we have $H^{d}_\fm(R)\cong E_R(k)$. By $(\ref{longcoh3})$, we deduce the following exact sequence
$$
0=H^{7}_{\fq_1}(R)\oplus H^{7}_{\fq_2}(R)\lo H^{7}_I(R)\lo H^{8}_{\fm}(R)\lo H^{8}_{\fq_1}(R)\oplus H^{8}_{\fq_2}(R)=0,
$$and the desired claim follows.
	
ii): There are two ideals $J_1,J_2$ of $R$ such that $\sqrt{J_1+J_2} = \fm$ and $\sqrt{J_1 \cap J_2} = \sqrt{I}$, but $\sqrt{J_1}$ and $\sqrt{J_2}$ are not equal to both $\fm$ and $\sqrt{I}$.  We apply the Hartshorne-Lichtenbaum vanishing theorem along with $(\ref{longcoh3})$ to get
$$
\cdots \lo H^{d-1}_{I}(R) \lo H^{d}_{\fm}(R) \lo H^{d}_{J_1}(R) \oplus H^{d}_{J_2}(R) =0.
$$
Recall that $H^{d}_\fm(R)\cong E_R(k)$. If $r\in\Ann(H^{d-1}_I(R))$, then $r$ annihilates any homomorphic image of $H^{d-1}_I(R)$, e.g., $r\in \Ann (E_R(k))=0$.
\end{proof}	
\end{remark}



\begin{example}
\label{32}
Let $k$ and $W(k)$ be as before. Define $R:=W(k)[|X_{i}~|~1 \leq i \leq 6|]$,  $I:=(X_1,X_2)\cap(X_3,X_4)\cap(X_5,X_6)$ and $A:=R/I$. Then the following assertions hold:
\begin{enumerate}
\item[i)]
$R$ is an unramified regular local ring and of dimension $7$.

\item[ii)]
$\Spec^{\circ}(A)$ is connected.

\item[iii)]
$H^{6}_I(R) = 0$.
\end{enumerate}
\end{example}

\begin{proof}
i): This is easy.
	
ii): This follows from Hartshorne's criteria, by showing that $\depth(A)\neq 1$. Instead, we show it more directly. Indeed, since $\Min(A)$ is equal to $\{(x_1,x_2),(x_3,x_4),(x_5,x_6)\}$ and for example $(x_1,x_2)+(x_3,x_4)$ is not primary to the maximal ideal, $\Theta_A$ is connected. In view of Lemma \ref{graph1}, $\Spec^{\circ}(A)$ is connected.

iii): Apply ii) along with Zhang's result.
\end{proof}

\subsection{An application of Zhang's result to calculating the invariant $q_{I}(R)$}

Finally, we apply Zhang's result ((SVT) for the unramified case) to calculate the cohomological dimension and the invariant $q_{I}(R)$, which we define below. Note that Proposition \ref{cohd} is the analogue of \cite[Proposition 3.1]{Var13}, and Proposition \ref{artininvariant} is the analogue of \cite[Proposition 3.2]{Var13} in mixed characteristic.

\begin{proposition}
\label{cohd}
Let $(R,\fm,k)$ be a $d$-dimensional unramified regular local ring of mixed characteristic, and assume that $I \subset R$ is a proper ideal with $\depth(R/I) \ge 2$. Then $H^{d-1}_I(R) = 0$. Suppose in addition that $R/I$ is 2-dimensional. Then $H^{d-2}_I(R) \neq 0$.
\end{proposition}

\begin{proof}
Since both depth and cohomological dimension behave well with respect to
completion, we may assume in addition that $R$ is complete. By Hartshorne's criteria \cite[Proposistion 2.1]{Har62}, $\Spec^\circ(R/I)$ is connected with respect to the Zariski topology. We apply this along with Zhang's result \cite[Theorem 1.4]{Zha21} to conclude that $H^{d-1}_I(R)=0$. From this, we have $\cd(I,R)\leq d-2$.

Now, suppose $R/I$ is $2$-dimensional.  By the catenary and the Cohen-Macaulay properties of $R$, we have $\grade(I,R)=\Ht(I,R)= d-2$. In view of
the cohomological characterization of grade, one has $H^{d-2}_I(R) \neq 0$.
\end{proof}


In his study of  $(SVT)$, Hartshorne \cite{HV} has introduced the invariant
$q_{I}(R)$ as the greatest integer $i$ such that $H^{i}_{I}(R)$ is not artinian. It is difficult to compute $q_{I}(R)$. Here, we present a sample, but we have no data about the cohomological dimension. In order to prove Proposition \ref{artininvariant}, we need the following lemma which was originally proved independently by   Auslander \cite[Lemma 3.4]{Aus}:

\begin{lemma}
\label{unramifiedlocalization}
Let $(R,\fm,k)$ be an unramified regular local ring of mixed characteristic and let $\fp$ be a prime ideal of $R$ such that $p \in \fp$. Then $R_\fp$ is also unramified.
\end{lemma}



\begin{proof}
Since we have an isomorphism $R_\fp/pR_\fp \cong (R/pR)_\fp$ and the localization of a regular local ring at a prime ideal is regular, we have that $R_\fp/pR_\fp$ is regular, and this implies that $p \in \fp R_\fp \backslash \fp^2 R_\fp$.
\end{proof}

\begin{remark}
The analogous result of Lemma \ref{unramifiedlocalization} is not necessarily true for ramified regular local rings. For example, put $R := C(k)[|X,Y,Z|]/(p-XY)$ and $A := R/pR = k[|X,Y,Z|]/(XY)$. Then $R$ is a ramified regular local ring.
Taking $\fp := (Y,Z) \subset R$, we obtain
$$
R_\fp/pR_\fp \cong A_\fp \cong k[|X,Y,Z|]_\fp/(XY) \cong k[|X,Y,Z|]_\fp/(Y).
$$
Since the last item is regular, we find that $p$ is a regular element on $R_\fp$. Hence $R_\fp$ is unramified.
\end{remark}

See \cite{Ion15} for the definition of almost Cohen-Macaulay rings.

\begin{proposition}
\label{artininvariant}
Let $(R,\fm,k)$ be a $d$-dimensional unramified regular local ring, and let  $\fp\in\Spec(R)$ be such that $R/\fp$ is a $3$-dimensional generalized and almost  Cohen-Macaulay ring. Then $q_\fp(R)=d-3$.
\end{proposition}

\begin{proof}
Since $R/ \fp$ is almost Cohen-Macaulay, we have $\depth(R/ \fp) \ge 2$.  Due to 
Proposition \ref{cohd} we know that  $\cd(I,R)\leq d-2$. It follows from the almost Cohen-Macaulay assumption that $R/\fp$ is Cohen-Macaulay over the punctured spectrum. By Grothendieck's vanishing theorem, $\fp\not\in\Supp(H^{d-2}_\fp(R))$.
Let $Q\in \V(\fp)$ be of height $d-2$. In view of the Hartshorne-Lichtenbaum vanishing theorem, $Q\not\in\Supp(H^{d-2}_\fp(R))$. Now, let $Q\in \V(\fp)$ be of height $d-1$.  We are going to show \begin{equation}
\label{ineq2}
 \cd(\fp R_Q,R_Q)\leq \dim(R_Q)-2=d-3. 
\end{equation}

 If $p \in Q$, then $R_Q$ is unramified by Lemma \ref{unramifiedlocalization}. Recall, from the generalized   Cohen-Macaulay assumption, that $R_Q/\fp R_Q$  is Cohen-Macaulay. So its depth is two. By another application of  Proposition \ref{cohd}, we  observe that $\cd(\fp R_Q,R_Q)\leq \dim(R_Q)-2=d-3.$ 
 If $p \notin Q$, then $R_Q$ is a regular local ring of equal characteristic $0$. From the above discussion and \cite[Proposition 3.1]{Var13}, we also  deduce (\ref{ineq2}). 
		According to (\ref{ineq2}), we conclude that  $Q\notin\Supp(H^{d-2}_\fp(R))$.   This implies that $\Supp(H^{d-2}_\fp(R))\subset \{\fm\}$. Recall that Bass numbers of local cohomology modules are finite by \cite[Theorem 1]{Lyu00}. Since $H^{d-2}_\fp(R)$ is supported at $\fm$ and its socle is finite, we deduce that  $H^{d-2}_\fp(R)$ is artinian. In sum, $H^{i}_\fp(R)$  is artinian for all $i>d-3$. In the light of Grothendieck's non-vanishing theorem, we know $H^{d-3}_\fp(R)_{\fp}\neq 0$, i.e., $\fp \in\Supp(H^{d-3}_\fp(R))$. From this we conclude that $q_\fp(R)= d-3$.
\end{proof}

\begin{corollary}
\label{qmod}
Adopt the notation of Proposition \ref{artininvariant} and let $M$ be a finitely generated $R$-module. Then $H^{i}_\fp(M)$ is artinian for all $i>d-3$.
\end{corollary}

\begin{proof}
This follows by a routine induction on $n:=\pd(M)$ which is finite by the regularity assumption of $R$.	
\end{proof}

\begin{corollary}
Let $(R,\fm,k)$ be a $d$-dimensional unramified regular local ring, and let $\fp\in\Spec(R)$. If $R/\fp$ is normal and of dimension three, then $q_\fp(R)= d-3$.
\end{corollary}

\begin{proof}
Recall that a localization of a normal domain is again normal. Also, in the light of Serre's characterization of normality, we know that normal rings satisfy Serre's condition $(S_2)$. From these facts, $R/\fp$ is a generalized and almost Cohen-Macaulay ring. It remains to apply Proposition \ref{artininvariant}.
\end{proof}

\begin{acknowledgement}
The authors are grateful to Shunsuke Takagi, Ryo Takahashi, and Ken-ichi Yoshida for useful comments. We also thank the referee for reading the paper thoroughly and providing valuable comments.
\end{acknowledgement}

\end{document}